\documentclass[11pt,reqno]{amsart}

\usepackage{float}
\usepackage{caption}
\usepackage{cases}

\usepackage{amsmath,amsfonts,amssymb, enumerate}

\usepackage[all]{xy}

\numberwithin{equation}{section}

\usepackage{hyperref}
\hypersetup{
	colorlinks,
	citecolor=red,
	filecolor=black,
	linkcolor=blue,
	urlcolor=black
}

\newtheorem{thm}{Theorem}[section]

\newtheorem{prp}[thm]{Proposition}

\newtheorem{mthm*}[thm]{Main Theorem}

\newtheorem{que}[thm]{Question}

\newcommand{\mbar}{\overline{M}}
\newcommand{\cp}{\mathbb{CP}}
\newcommand{\ev}{\textnormal{ev}}
\newcommand{\gw}[3]{\textnormal{GW}^{#1}_{#2} \left( #3 \right)}
\newcommand{\X}{\mathfrak{X}}
\newcommand{\mbarfib}[1]{\mbar_{0,#1}^{\textnormal{Fib}}}

\usepackage{scalerel}

\makeatletter
\renewcommand{\tocsection}[3]{%
\indentlabel{\@ifnotempty{#2}{\bfseries\ignorespaces#1 #2\quad}}\bfseries#3}
\renewcommand{\tocsubsection}[3]{%
  \indentlabel{\@ifnotempty{#2}{\ignorespaces#1 #2\quad}}#3}

\newcommand\@dotsep{4.5}
\def\@tocline#1#2#3#4#5#6#7{\relax
  \ifnum #1>\c@tocdepth 
  \else
    \par \addpenalty\@secpenalty\addvspace{#2}%
    \begingroup \hyphenpenalty\@M
    \@ifempty{#4}{%
      \@tempdima\csname r@tocindent\number#1\endcsname\relax
    }{%
      \@tempdima#4\relax
    }%
    \parindent\z@ \leftskip#3\relax \advance\leftskip\@tempdima\relax
    \rightskip\@pnumwidth plus1em \parfillskip-\@pnumwidth
    #5\leavevmode\hskip-\@tempdima{#6}\nobreak
    \leaders\hbox{$\m@th\mkern \@dotsep mu\hbox{.}\mkern \@dotsep mu$}\hfill
    \nobreak
    \hbox to\@pnumwidth{\@tocpagenum{\ifnum#1=1\bfseries\fi#7}}\par
    \nobreak
    \endgroup
  \fi}
\AtBeginDocument{%
\expandafter\renewcommand\csname r@tocindent0\endcsname{0pt}
}
\def\l@subsection{\@tocline{2}{0pt}{2.5pc}{5pc}{}}
\makeatother

\begin{document}

\title[Counting rational curves with an $m$-fold point]{Counting rational curves with an $m$-fold point}

\author[I. Biswas]{Indranil Biswas}

\address{Department of Mathematics, Shiv Nadar University, NH91, Tehsil
Dadri, Greater Noida, Uttar Pradesh 201314, India}

\email{indranil.biswas@snu.edu.in, indranil29@gmail.com}

\author[C. Chaudhuri]{Chitrabhanu Chaudhuri}
\address{School of Mathematical Sciences, National Institute of Science Education and Research, HBNI, Bhubaneswar, Odisha- 752 050, India.}

\email{chitrabhanu@niser.ac.in }

\author[A. Choudhury]{Apratim Choudhury}
\address{School of Mathematical Sciences, National Institute of Science Education and Research, HBNI, Bhubaneswar, Odisha- 752 050, India.}

\email{apratim.choudhury@niser.ac.in}

\author[R. Mukherjee]{Ritwik Mukherjee}
\address{School of Mathematical Sciences, National Institute of Science Education and Research, HBNI, Bhubaneswar, Odisha- 752 050, India.}

\email{ritwikm@niser.ac.in}

\author[A. Paul]{Anantadulal Paul}
\address{International Center for Theoretical Sciences, Sivakote, Bangalore, 560089, India}

\email{anantadulal.paul@icts.res.in}

\subjclass[2010]{14N35, 14J45, 53D45}

\keywords{Enumerative Geometry, Singularity, Gromov-Witten Invariants, WDVV equation}

\date{}

\begin{abstract}
We obtain a recursive formula for the number of rational curves of degree $d$ in
$\mathbb{CP}^2$, that pass through $3d+1-m$ generic points and that have an $m$-fold
singular point. The special case of counting curves with a triple point was solved
earlier by other authors. 
We obtain the formula by considering a family version of Kontsevich's 
recursion formula, in contrast to the excess intersection theoretic approach of others.
A large number of low degree cases have been worked out explicitly.
\end{abstract}

\maketitle

\tableofcontents

\section{Introduction}

The enumerative geometry of rational curves in projective spaces is a rich subject with a long history. It 
has been studied by mathematicians for more than a century. Using the concept of moduli space of stable 
maps and Gromov-Witten Invariants, Kontsevich solved the following question: 
\begin{que}
\label{qu}
How many rational degree $d$ curves are there in $\mathbb{CP}^2$ that pass through $3d-1$ generic points? 
\end{que}
Kontsevich's solution to 
the above problem 
is simply ingenious. The idea is to look at 
$\overline{M}_{0,4}$, 
the moduli space of four marked points on a sphere. This space is isomorphic to $\mathbb{CP}^1$.
Hence, any two points determine the same divisor. The next step is to pull back the divisors 
on $\overline{M}_{0,4}(\mathbb{P}^2, d)$, 
the moduli space of stable maps with four marked points. 
After that, intersect the two divisors with appropriate cycles of complementary dimension, 
which gives an equality of numbers resulting in 
the famous Kontsevich's recursion formula (\cite{K.M}). 

Kontsevich's solution to question \ref{qu} is arguably one of the most remarkable achievements of the theory of 
moduli space of stable maps and Gromov-Witten invariants. Since then, mathematicians have been interested in the following 
question: 
\begin{que}
\label{qu_sing}
How many rational degree $d$ curves are there in $\mathbb{CP}^2$ that pass through the correct number of generic points 
and have a singularity of a certain type?
\end{que}

Question \ref{qu_sing} has been extensively studied when the singularity is a triple point.
Those who have studied this question, include among others 
Katz, Qin and Ruan (\cite{Katz-Triple}), Ravi Vakil (\cite{Vakil_arxiv}), 
Ziv Ran (\cite{Ran_Triple}, \cite{Ran}), Joachim Kock (\cite{Kock_Cusp})
and subsequently Aleksey Zinger (\cite{zin-triple}). 
The problem has been practically unapproachable for the general case of an $m$-fold singular point, when $m>3$. 

With this background, we now state the main result of our paper. 

\begin{mthm*}
\label{main_thm_statement}
Let $m$ be an integer greater than or equal to $3$. 
We have a recursive formula 
to compute the number of rational degree $d$ curves in $\mathbb{CP}^2$ 
passing through $3d+1-m$ generic points and that has an $m$-fold point.
The formula is given by equations \eqref{boundary} and \eqref{main_formula_bl} 
and the base case of the recursion is 
given by Proposition \ref{base_case_prpn}. 
\end{mthm*}
We have written a mathematica program
to implement the above formula. 
The program is available on our web page
\[ \textnormal{\url{https://www.sites.google.com/site/ritwik371/home}}. \]

Before we get into the technical details of how we solve the above problem, let us make a few remarks 
to put things in perspective. 
One of the most non-trivial parts of enumerative geometry 
is computing the degenerate contribution to the 
Euler class (also called excess intersection theory). 
While
computing the degenerate contribution to the Euler class is undoubtedly
a highly nontrivial task, it may be viewed as a bug rather than a feature. 
If the enumerative question can be recast in a suitable way, 
this issue can be avoided and excess intersection theory may be
bypassed.

In this paper, we show that the answer to Question \ref{qu_sing} can be obtained by 
replicating Kontsevich's idea. The novelty of our approach is to \textit{setup} the problem 
in such a way so that that the problem is drastically simplified. While earlier 
approaches to Question \ref{qu_sing} have succeeded when the singularity is a triple 
point, the approach in this paper directly solves the question for any $m$-fold point and
that too with substantially less efforts. This is the main achievement of this paper. 

\section{The intuitive idea behind our approach}

To keep this discussion simple, let us take $m\,=\,3$, i.e., the rational curves with 
a triple point are being counted. More precisely, we wish to compute the number of
rational degree $d$ 
curves, passing through $3d-2$ points and that have a triple point. Let us for a moment try 
to answer a different, but potentially easier question. How many rational degree $d$ curves 
are there through $3d-4$ points, that have a triple point at a given point?
This question can indeed be solved very easily by using no excess 
intersection theory. The desired number is simply the number of rational curves in 
$\mathbb{CP}^2$ blown up at a point, representing the class $dL-3E$ and that passes through 
$3d-4$ points. 
We will prove this fact
in Section \ref{correspondence}. 
It follows from a correspondence result in classical Algebraic Geometry 
which is studied by the author in \cite{Gath_blow_up}. 
For the moment 
we will accept this assertion.

Let us now return to the original problem, 
namely how to enumerate rational curves with a triple point somewhere (as opposed to the 
triple point confined to a fixed location). The idea is very simple: allow that triple 
point to vary in a family. More precisely, consider a fiber bundle over $\mathbb{CP}^2$, 
whose fiber over each point $q$ is the moduli space of stable maps into $\mathbb{CP}^2$ blown up 
at $q$, representing the class $dL-3E$ (where $L$ and $E$ denote the class of a line and 
the exceptional divisor respectively). The main point now is to execute intersection theory on 
this moduli space. But that can be achieved very easily. It is conceptually no more 
difficult than the original Kontsevich's recursion formula --- one simply needs to observe 
that $\overline{M}_{0,4}$ is isomorphic to ${\mathbb C}{\mathbb P}^1$ and hence any
two points define the same 
divisor. Pulling back those divisors and intersecting with appropriate cycles, produces us the 
desired recursion formula. That is it. Here we have carried out this approach; it may be
mentioned that the numbers in \cite{zin-triple} have been recovered. In fact, our 
approach directly gives the characteristic number of rational degree $d$ curves with an 
$m$-fold point; we simply have to count curves in the class $dL-mE$.

\section{A family version of Kontsevich's recursion formula}\label{family_kontsevich}

\subsection{The moduli space}
Let $X\,:= \,\mathbb{CP}^2 \times \mathbb{CP}^2.$ 
In order to distinguish each of the factors, the second copy of $\mathbb{CP}^2$ will be
denoted by $\mathbb{CP}^2_{\textnormal{Base}}$ (the reason for the notation will become 
clear in a moment). Hence, 
\[
X\,=\, \mathbb{CP}^2 \times \mathbb{CP}^2_{\textnormal{Base}}. 
\]
Define $\mathfrak{X}$ to be $X$ blown up along the diagonal. Notice that 
\[
\pi\,:\,\mathfrak{X} \,\longrightarrow \,\mathbb{CP}^2_{\textnormal{Base}}
\]
is a holomorphic fiber bundle over $\mathbb{CP}^2_{\textnormal{Base}}$, whose fiber over each 
point $q\,\in\, \mathbb{CP}^2_{\textnormal{Base}}$ is the blow up $\mathbb{CP}^2_q$ of
$\mathbb{CP}^2$ at the point $q$.

The pull-back, to $\mathfrak{X}$, of hyperplane classes of $\mathbb{CP}^2$ and
$\mathbb{CP}^2_{\textnormal{Base}}$ will be denoted by $L$ and $a$ 
respectively. Furthermore, the exceptional divisor in $\mathfrak{X}$ will be
denoted by $E$. Therefore, the cohomology ring of $\mathfrak{X}$ is 
generated by the classes $L$, $a$ and $E$, i.e., 
\[
H^*(\mathfrak{X},\, {\mathbb Z}) \,=\, \langle L,\, a,\, E \rangle.
\]

Given an $\alpha = (d,m) \,\in\, \mathbb{Z}^2$, such that $\alpha \neq (0,0)$, 
define the fiber bundle 
\begin{equation}\label{e1}
\pi_F\,:\,\overline{M}_{0,r}^{\textnormal{Fib}}(\mathfrak{X}, \alpha)
\,\longrightarrow \,\mathbb{CP}^2_{\textnormal{Base}}
\end{equation}
as follows: the fiber over each point $q$ is the moduli space of genus zero, $r$ pointed, 
stable maps into $\mathbb{CP}^2_q$ representing the class 
\[\beta_q\,:= \,dL_q-mE_q,\] 
where $L_q$ and $E_q$ denote the class of a line and exceptional divisor in 
$\mathbb{CP}^2_q$ respectively. In other words, the fiber is 
\[
\pi_{F}^{-1}(q)\,=\, \overline{M}_{0,r}(\mathbb{CP}^2_q, \,\beta_q).
\]
Using the fact that the dimension of the total space of a fiber bundle is the dimension of
the fiber plus the dimension of the base, we conclude that the 
dimension of the moduli space $\mbarfib{r}(\X,\alpha)$ is given by 
\begin{equation*}
\dim \mbarfib{r}(\X,\alpha)\, =\, \dim \mbar_{0,r}(\cp^2_q, \beta_q) + \dim \cp^2
\,=\, (3d-1-m)+2 \,=\, 3d+1-m.
\end{equation*}

As customary, the moduli space $\overline{M}_{0,r}^{\textnormal{Fib}}
(\mathfrak{X},\,\alpha)$ admits $r$ natural evaluation maps 
into $\mathfrak{X}$. 

\subsection{Recursive formula}

Let $a \,\in\, H^2(\cp^2_{\text{Base}})$ be the hyperplane class. Define 
\begin{align}
N_{\alpha}(r, \,\theta)&\,:=\, \big[\overline{M}_{0,r}^{\textnormal{Fib}}(\mathfrak{X},
\,\alpha)\big] \cdot \textnormal{ev}_1^*(L^2) \ldots 
\textnormal{ev}_r^*(L^2)\cdot \pi_F^*(a^{\theta}). \label{N_alpha_defn}
\end{align}
We formally declare $N_{\alpha}(r, \theta)$ to be zero, unless the right hand side of \eqref{N_alpha_defn} dimensionally makes 
sense. In other words, 
\begin{align}
N_{\alpha}(r, \,\theta) & \,= \,0 \qquad \textnormal{ if} \qquad 3d+1-m
\,\neq\, r+\theta. \label{dim_condition}
\end{align}
This is because $\dim \mbarfib{r}(\X,\alpha) = 3d+1-m$, and codimension of the cycle 
$\ev_1^*(L^2) \cdots \ev_r^*(L^2) \cdot \pi_F^*(a^\theta)$ is $2r+\theta$. Hence, 
$N_{\alpha}(r,\, \theta)$ should be declared to be zero, unless $3d+1-m+r \,=\, 2r+\theta$.

We claim that for 
\begin{align}
\alpha&\,=\, (d,\,m), \qquad r\,=\,3d+1-m\qquad \textnormal{and} \qquad
\theta\,=\,0, \label{alpha_value}
\end{align}
$N_{\alpha}(r,\, \theta)$ in \eqref{N_alpha_defn}
is the number of rational curves of degree $d$ in $\mathbb{CP}^2$, passing through $3d+1-m$ generic points and having an $m$ fold point singularity. 
This claim will be proved in Section \ref{correspondence}.

To state the recursive formula for $N_{\alpha}(r, \theta)$, given 
$\alpha_1,\, \alpha_2 \,\in\, \mathbb{Z}^2$ and $r_1,\, r_2,\, \theta \,\in \,\mathbb{Z}_{\geq 0}$
we introduce the quantities 
\begin{align} \label{boundary}
\mathcal{B}_{\alpha_1, \alpha_2}(r_1,\, r_2,\,\theta) := \ &
N_{\alpha_1}(r_1,\, \theta+2) N_{\alpha_2}(r_2,\, 0)\nonumber \\
 & + N_{\alpha_1}(r_1,\, \theta+1) N_{\alpha_2}(r_2,\, 1) 
+ N_{\alpha_1}(r_1,\, \theta) N_{\alpha_2}(r_2,\, 2). 
\end{align}

For the ease of notation define
\begin{align*}
l\,:=\, (1,\,0) \,\in \,\mathbb{Z}^2 \qquad \textnormal{and} \qquad e
\,:=\, (0,\,-1) \,\in\, \mathbb{Z}^2. 
\end{align*}
Also consider the following product in $\mathbb{Z}^2$: 
\[
(d_1,\,m_1)\cdot(d_2,\,m_2) \,=\, d_1 d_2 - m_1 m_2.
\]

\begin{thm} \label{rf_main}
For $\alpha \,\in \,\mathbb{Z}^2 \setminus \{(0,\,0)\}$ and $r,\,\theta \,\in \,\mathbb{Z}_{\geq 0}$,
the following recursive formula holds:
\begin{align}
N_{\alpha}(r, \,\theta) \,=\, \sum_{\substack{r_1+r_2 = r-3, \\ \alpha_1+\alpha_2 = \alpha}}
\binom{r-3}{r_1} (\alpha_1 \cdot \alpha_2) 
(\alpha_1\cdot l) 
& \Big( \mathcal{B}_{\alpha_1, \alpha_2}(r_1+1,\, r_2+1,\, \theta) (\alpha_2 \cdot l) \nonumber \\ 
& - \mathcal{B}_{\alpha_1, \alpha_2}(r_1,\, r_2+2,\, \theta)(\alpha_1 \cdot l) \Big). \label{main_formula_bl}
\end{align}
\end{thm}

To compute $N_{\alpha}(r,\, \theta)$ using the recursion formula in Theorem \ref{rf_main} we have to
know the base cases. For this, denote by
\begin{equation}\label{e2}
n_{\alpha}^{\mathbb{CP}^2_1}(r)
\end{equation}
the number of rational curves in $\mathbb{CP}^2$ blown up at one point, representing the 
class $dL-mE$ and passing through $r$ generic points. As before, unless $r\,=\,3d-1-m$, we formally 
define this number to be zero. 
These numbers can be computed from the formula given in \cite{Rahul_Gottsche}.
We are now ready to state all the necessary base cases of the recursive formula \eqref{rf_main}; this will
be stated as the following proposition. 

\begin{prp}
\label{base_case_prpn}
Let $N_{\alpha}(r,\theta)$ be defined by equation \ref{N_alpha_defn}. 
Then the following are true:
\begin{subnumcases}{N_{\alpha}(r,\, \theta)\,=\,} 
 0 & \textnormal{if} ~~$r\,\neq\, 3d+1-m-\theta$, \label{base_case_rec1_pt5} \\
\textnormal{$n_{\alpha}^{\mathbb{CP}^2_1}(r)$} & \textnormal{if} ~~$\theta\,=\,2$, \label{base_case_rec1_pt1} \\ 
 0 & \textnormal{if} ~~$\theta\,\geq\, 3$, \label{base_case_rec1_pt4} \\
 0 & \textnormal{if} ~~$d\,<\,0$, \label{base_case_rec1_pt2}\\
 0 & \textnormal{if} ~~$m\,<\,0$, ~~\textnormal{unless} ~~$d\,=\,0$, $m\,=\,
-1$, ~~\textnormal{and} ~~$\theta\,=\,2$, \label{base_case_rec1_pt3} \\
 0 & \textnormal{if} ~~$d\,\geq\, 2$ ~~\textnormal{and}~~ $m\,\geq\, d$. \label{base_case_rec1_pt6}
\end{subnumcases} 
Furthermore, 
\begin{align}
N_{l}(4,\, 0) & \,=\, 0, \label{base_case_rec2_pt3} \\
N_{l}(3,\,1) & \,= \,0,    \label{base_case_rec2_pt4} \\    
N_{l-e}(3,\,0) & \,= \,0,  \label{base_case_rec2_pt5} \\  
N_{l-e}(2,\,1) & \,=\, 1,  \label{base_case_rec2_pt6} \\  
N_{l-me} (r, \,\theta)  & \,=\, 0, \qquad \forall\,\, ~~m \,\geq\, 2. \label{base_case_rec2_pt7} 
\end{align}
\end{prp}

We note that Theorem \ref{rf_main} combined with Proposition \ref{base_case_prpn}, 
enables us to compute $N_{\alpha}(r,\, \theta)$ for all $\alpha$, $r$ and $\theta$. The rest 
of this section is devoted to the proofs of these results.

\subsection{Proofs of the recursion and initial cases}

\begin{proof}[Proof of Theorem \ref{rf_main}]
As in Kontsevich's recursion formula, we consider the forgetful morphism
\begin{align*}
\pi\,:\,\overline{M}_{0,r+1}^{\textnormal{Fib}}(\mathfrak{X},\, \alpha)
\,\longrightarrow\, \overline{M}_{0,4}.
\end{align*}
The moduli space $\mbar_{0,4}$ parametrizes stable genus zero curves with four marked points.
In $\mbar_{0,4}$ we have the boundary divisors $D(12|34)$,\, $D(13|24)$ and $D(14|23)$.
Define the cycle $\mathcal{Z} \in H^{4r+2\theta}\Big(\mbarfib{r+1}(\X, \alpha)\Big)$ as 
\begin{align*}
\mathcal{Z}&\,:=\, 
\Big(\textnormal{ev}_1^{*}(L) \cdot \textnormal{ev}_2^{*}(L)) 
\cdot \textnormal{ev}_3^{*}(L^2) \cdot \textnormal{ev}_4^{*}(L^2)\Big) \cdot \Big(\textnormal{ev}_5^{*}(L^2) 
\ldots \textnormal{ev}_{r+1}^*(L^2) \Big) \cdot \pi_F^*a^{\theta}.
\end{align*}

We claim that 
\begin{align}
[\pi^*D(12|34)]\cdot \mathcal{Z} & \,=\, N_{\alpha}(r,\,\theta) +
\sum_{\substack{r_1+r_2 \,=\, r-3, \\ 
\alpha_1+\alpha_2 = \alpha}} \binom{r-3}{r_1}
(\alpha_1\cdot l)^2 \Big(\mathcal{B}_{\alpha_1, \alpha_2}(r_1,\, r_2+2,\, \theta)
(\alpha_1 \cdot \alpha_2) \Big)  \label{1234}
\end{align}
and
\begin{align} 
[\pi^*D(13|24)]\cdot \mathcal{Z} & \,=\, \sum_{\substack{r_1+r_2 = r-3, \\ 
\alpha_1+\alpha_2\,=\, \alpha}} \binom{r-3}{r_1} 
(\alpha_1\cdot l) \cdot (\alpha_2 \cdot l)
\Big( \mathcal{B}_{\alpha_1, \alpha_2}(r_1+1,\, r_2+1,\, \theta) (\alpha_1 \cdot \alpha_2)\Big). \label{1324}
\end{align}

We will justify both the assertions shortly. First observe that 
\begin{align}
[\pi^*D(12|34)]\cdot \mathcal{Z} & \,\,= \,\,[\pi^*D(13|24)]\cdot \mathcal{Z}. \label{1234_is_eq_to_1324}
\end{align}
From \eqref{1234_is_eq_to_1324}, \eqref{1234} and \eqref{1324} we get the required expression for
$N_{\alpha}(r,\theta)$.

Claims \eqref{1234} and \eqref{1324} will now be proved.

First, let us make a couple of 
definitions. For $\alpha_i \,=\, (d_i,\, m_i) \,\in\, \mathbb{Z}^2$, a stable map 
$u \,\in\, \mbarfib{r+1}(\X,\,\alpha)$ to be called of type $(\alpha_1, \,\alpha_2)$ if the following holds:
\begin{itemize}
\item The domain has two components $C_1$ and $C_2$ meeting at a node. Each component
is isomorphic to $\cp^1$ and we denote $u_i \,=\, u |_{C_i}$.

\item The image of both $u_1$ and $u_2$ lie inside $\mathbb{CP}^2_q$ for some $q$. 

\item The map $u_1$ represents the class $d_1 L_q - m_1 E_q$ and the map
$u_2$ represents the class $d_2 L_q - m_2 E_q$.
\end{itemize}

Next, we call a stable map of type $(\alpha_1, \,\alpha_2)$ 
to be decorated with $(r_1,\, r_2)$ marked points if there are $r_1$ marked points on the 
first curve and $r_2$ marked points on the 
second curve. Pictorially, this can be represented as follows: 
\begin{figure}[H]
\vspace*{0.2cm}
\begin{center}\includegraphics[scale=1]{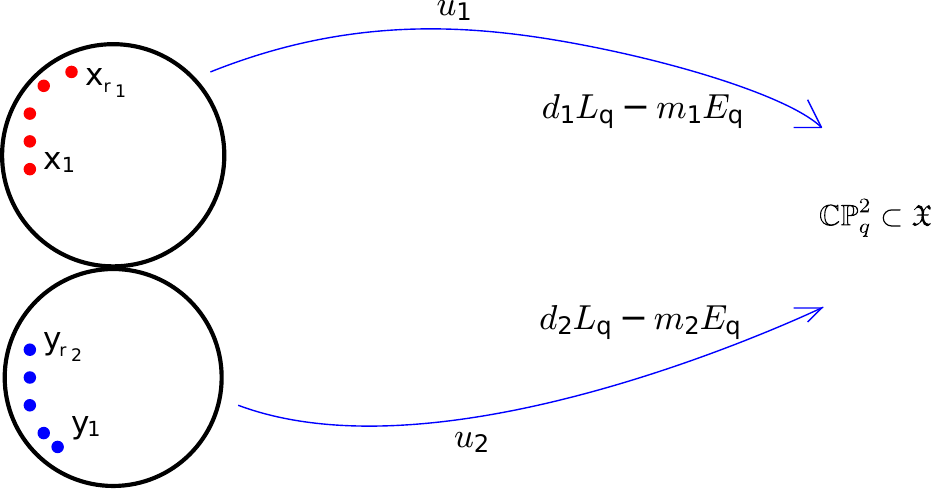}\vspace*{-0.2cm}\end{center}
\end{figure}

Let us now explain the geometric significance of the number defined by
$\mathcal{B}_{\alpha_1, \alpha_2}(r_1,\, r_2,\,\theta)$, 
namely the right hand side of \eqref{boundary}.
Consider the product of the bundles
\begin{align}\label{e6}
\Pi\,=\, (\pi_1,\, \pi_2)\,:\,\overline{M}_{0,r_1}^{\textnormal{Fib}}(\mathfrak{X},\, \alpha_1)
\times \overline{M}_{0,r_2}^{\textnormal{Fib}}(\mathfrak{X},\, \alpha_2)
&\,\longrightarrow\, \mathbb{CP}^2_{\textnormal{Base}_1} \times \mathbb{CP}^2_{\textnormal{Base}_2}
\end{align}
defined in \eqref{e1}. If we take two stable maps $u_i \,\in\, \mbarfib{r_i}(\X,\alpha_i)$, $i\,=\,1,\,
2$, their images might not be in the same fiber. So we have to look at the pull back of the diagonal
of $\mathbb{CP}^2_{\textnormal{Base}_1} \times \mathbb{CP}^2_{\textnormal{Base}_2}$ by
$\Pi$. Hence the cycle 
\begin{equation}\label{e5}
\Pi^*(\Delta_{\mathbb{CP}^2_{\textnormal{Base}_1} \times \mathbb{CP}^2_{\textnormal{Base}_2}}) 
\,\, \subset\,\, \overline{M}_{0,r_1}^{\textnormal{Fib}}(\mathfrak{X},\, \alpha_1)
\times \overline{M}_{0,r_2}^{\textnormal{Fib}}(\mathfrak{X},\, \alpha_2)
\end{equation}
represents the closure of the stable maps of type $(\alpha_1, \,\alpha_2)$ decorated with 
$(r_1, \,r_2)$ points. We now intersect this cycle in \eqref{e5} with 
\begin{align*}
\pi_1^*\Big(\textnormal{ev}_1^*(L^2) \ldots 
\textnormal{ev}_{r_1}^*(L^2)\cdot \pi_F^*(a^{\theta})\Big) \cdot 
\pi_2^*\Big(\textnormal{ev}_1^*(L^2) \ldots 
\textnormal{ev}_{r_2}^*(L^2)\Big),
\end{align*}
where $\pi_1$ and $\pi_2$ are the projection maps in \eqref{e6}. Note that this intersection will 
precisely give the right hand side of equation \eqref{boundary}.
Geometrically, this intersection gives the number of stable maps of type $(\alpha_1, \,\alpha_2)$ 
decorated with $(r_1,\, r_2)$ points such that the following conditions hold:
\begin{itemize}
\item If the image of the stable map lies in $\mathbb{CP}^2_q$, then 
the $r_1$ marked points on the first curve intersect $r_1$ generic points in $\mathbb{CP}^2_q$.
This is because of the presence of the factor of $\pi_1^*\Big(\textnormal{ev}_1^*(L^2) \ldots 
\textnormal{ev}_{r_1}^*(L^2)\Big)$ while computing the intersection number.

\item Similarly, if the image of the stable map lies in $\mathbb{CP}^2_q$, 
then the $r_2$ marked points on the second curve intersect $r_2$ generic points in $\mathbb{CP}^2_q$. 
This is because of the presence of the factor of $\pi_2^*\Big(\textnormal{ev}_1^*(L^2) \ldots 
\textnormal{ev}_{r_2}^*(L^2)\Big)$ while computing the intersection number.

\item Finally, because of the presence of the factor of $\pi_F^*(a^{\theta})$ while computing the 
intersection number, there is a possible restriction on the $\mathbb{CP}^2_q$ where the image of the 
stable map is allowed to lie. If $\theta\,=\,0$, then there is no restriction. If $\theta\,=\,2$, then the 
image of the stable map can only lie in a fixed $\mathbb{CP}^2_q$. The point $q$ will be representing the 
class $a^2$. If $\theta\,=\,1$, then the image of the stable map can only lie in a
$\mathbb{CP}^2_q$, where 
$q$ lies on a generic line. In this case, the line will be representing the class $a$.
\end{itemize}

We are now ready to justify equation \eqref{1234}. 
The left hand side of equation \eqref{1234} can be represented by the following picture:
\begin{figure}[H]
\vspace*{0.2cm}
\begin{center}\includegraphics[scale = .5]{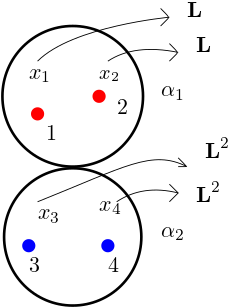}\vspace*{-0.2cm}\end{center}
\end{figure}
Intersect the above space with $\mathcal{Z}$. First assume $\alpha_1$ and $\alpha_2$ are
both nonzero. Consider the marked points 
\[x_5,\,\, x_6,\,\, \cdots,\,\, x_{r+1}.\] 
Notice that there are $r-3$ of these marked points in $\mathcal{Z}$ (here we are not counting 
the first four marked points, namely $x_1,\, x_2,\, x_3$ and $x_4$). Suppose $r_1$ of 
these marked points are on the $\alpha_1$ component and $r_2\,= \,r-3-r_1$ of these marked 
points on the $\alpha_2$ component. When we intersect that configuration with 
$\mathcal{Z}$, we are computing the cardinality of the set $S$ multiplied by 
$(\alpha_1 \cdot \alpha_2) (\alpha_1\cdot l)^2$, where $S$ is the following set: 
\begin{itemize}
 \item Stable maps of type 
$(\alpha_1, \,\alpha_2)$ decorated with $(r_1,\, r_2+2)$ marked points. This is because of the presence of 
$\textnormal{ev}_3^{*}(L^2) \cdot \textnormal{ev}_4^{*}(L^2)$ in the definition of $\mathcal{Z}$.

\item The image of the stable map lies inside $\mathbb{CP}^2_q$, where $q$ lies in a generic cycle representing $a^{\theta}$.
\end{itemize}
The reason we multiply by $(\alpha_1 \cdot \alpha_2)$ is because there are $(\alpha_1 \cdot \alpha_2)$ many choices for the nodal point of the 
stable map, each such map is counted separately in the moduli space of curves. Finally, we multiply by 
$(\alpha_1\cdot l)^2$ due to the presence of 
$\textnormal{ev}_1^{*}(L) \cdot \textnormal{ev}_2^{*}(L)$ in the definition of $\mathcal{Z}$. 
We note that the cardinality of $S$ is given by 
\begin{align}
\mathcal{B}_{\alpha_1, \alpha_2}(r_1,\, r_2+2,\, \theta). \label{bubble_1234_cardinality_temp} 
\end{align}
Hence, the cardinality of $S$ multiplied by 
$(\alpha_1 \cdot \alpha_2) (\alpha_1\cdot l)^2$ 
is given by 
\begin{align}
\mathcal{B}_{\alpha_1, \alpha_2}(r_1,\, r_2+2,\, \theta) (\alpha_1 \cdot \alpha_2) (\alpha_1\cdot l)^2. \label{bubble_1234_cardinality} 
\end{align}
We note that $\binom{r-3}{r_1}$ times the expression in equation \eqref{bubble_1234_cardinality} is precisely equal to the summand of the 
second term in equation \eqref{1234}. 

Next, suppose that $\alpha_1\,=\,0$ and $\alpha_2\,= \,\alpha$. This means
that the $\alpha_1$ component is a constant map. 
The entire second component, which is of degree $\alpha$, passes through the $r-3$ points, plus 
an additional two more points 
coming from $\textnormal{ev}_3^*(L^2)$ and $\textnormal{ev}_4^*(L^2)$ and one more point, that 
comes from intersecting 
$\textnormal{ev}_1^*(L)$ and $\textnormal{ev}_2^*(L)$. Hence, the cardinality of that set 
is $N_{\alpha}(r, \theta)$. This gives us the right hand side of equation \eqref{1234}.

Let us now justify equation \eqref{1324}. Note that 
the left hand side of equation \eqref{1324} can be represented by the following picture:
\begin{figure}[H]
\vspace*{0.2cm}
\begin{center}\includegraphics[scale = .7]{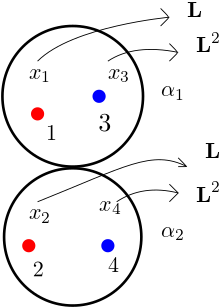}\vspace*{-0.2cm}\end{center}
\end{figure}
The argument is now same as before. We note that the configuration $\alpha_1\,=\, 0,\, 
\alpha_2\,=\, \alpha$ and $\alpha_1\,=\, \alpha, \, \alpha_2\,= \,0$ will not occur, since
a generic line and a generic point cannot intersect at a fiber. This concludes the proof of the recursive formula.
\end{proof}

Finally let us prove Proposition \ref{base_case_prpn} which establishes the base 
cases for the recursion.

\begin{proof}[Proof of Proposition \ref{base_case_prpn}]
First note that equation \eqref{base_case_rec1_pt5} is simply 
rewriting equation \eqref{dim_condition} again.

Next, let us justify equation \eqref{base_case_rec1_pt1}. Consider the cycle 
$\pi_F^*(a^{2})$ inside $\overline{M}_{0,r}^{\textnormal{Fib}}(\mathfrak{X}, \alpha)$.
This cycle is represented by curves mapping into $\mathbb{CP}^2_q$ (representing the homology class $dL_q-mE_q$) 
for some fixed $q$. 
The point $q$ will be representing the class $a^2$. 
Hence, $N_{\alpha}(r,2)$, which is given by 
\[ \big[\overline{M}_{0,r}^{\textnormal{Fib}}(\mathfrak{X},
\,\alpha)\big] \cdot \textnormal{ev}_1^*(L^2) \ldots 
\textnormal{ev}_r^*(L^2)\cdot \pi_F^*(a^{2}), \]
is the number of curves in $\mathbb{CP}^2_q$ representing the class $dL_q-mE_q$ and passing through 
$r$ generic points. That number is precisely $n_{\alpha}^{\mathbb{CP}^2_1}(r)$.

Equation \eqref{base_case_rec1_pt4} follows from the fact that $\pi_F^*(a^\theta)$ is the
zero class if $\theta\,\geq\, 3$ 
(since $a$ is the hyperplane class of $\mathbb{CP}^2_{\textnormal{Base}}$). 
Equation \eqref{base_case_rec1_pt2} follows from the fact that on $\mathbb{CP}^2_1$, 
there are no curves in the class $dL-mE$ if $d$ is negative (regardless of the value of $m$). 

Equation \eqref{base_case_rec1_pt3} will require a bit more effort to justify. 
First note that on $\mathbb{CP}^2_1$, 
there are no irreducible curves in the class $dL-mE$ if $d$ is positive and $m$ is negative. If $d$ is
negative, then from equation \eqref{base_case_rec1_pt2} 
we conclude that $N_{\alpha}(r, \theta)$
is zero. If $d\,=\,0$, then $N_{\alpha}(r, \theta)$ 
can potentially be nonzero. The cases $\theta\,=\,2$ and $\theta\,\geq\, 3$ are dealt by equations
\eqref{base_case_rec1_pt1} and \eqref{base_case_rec1_pt4} respectively. 
Hence, what remains is to consider the cases $\theta
\,=\,1$ or $0$. First, assume that $m\,=\,-1$. 
By equation \eqref{dim_condition}, in order for 
$N_{\alpha}(r, \theta)$ to be nonzero, we require
that $r+\theta\,=\,2$. If $\theta\,=\,1$ or $0$, then $r\,=\,1$ and $2$ respectively; in both cases
$r$ is nonzero. If $m\,\leq\,-2$, then by equation \eqref{dim_condition}, in order for 
$N_{\alpha}(r, \theta)$ to be nonzero, $r$ has to be nonzero (since $\theta\,\leq\, 2$). 
Hence, in all the cases, $r$ is nonzero, which is all that matters for the 
argument we are about to make. 
 
Let us now interpret the number
$N_{(d,m)}(r,\theta)$ geometrically, where $d\,=\,0$, $m\,\leq\,-1$ 
and $r$ is nonzero. The number $N_{(0,m)}(r,\theta)$ is given by 
\[ \big[\overline{M}_{0,r}^{\textnormal{Fib}}(\mathfrak{X},
\,(0,m))\big] \cdot \textnormal{ev}_1^*(L^2) \ldots 
\textnormal{ev}_r^*(L^2)\cdot \pi_F^*(a^{\theta}). \]
We claim that the above number is zero.

First, consider the cycle 
\[ \big[\overline{M}_{0,r}^{\textnormal{Fib}}(\mathfrak{X},
\,(0,m))\big] \]
This cycle denotes the space of rational curves into $\mathfrak{X}$, whose image lies in a $\mathbb{CP}^2_q$ and which represents the 
class $-mE_q$. Hence,
\[ \big[\overline{M}_{0,r}^{\textnormal{Fib}}(\mathfrak{X},
\,(0,m))\big] \cdot \textnormal{ev}_1^*(L^2) \ldots 
\textnormal{ev}_r^*(L^2)\cdot \pi_F^*(a^{\theta}), \]
denotes the number of rational curves into $\mathfrak{X}$, whose image lies in a $\mathbb{CP}^2_q$, 
representing the class $-mE_q$ and which passes through $r$ generic points in $\mathbb{CP}^2_q$. If 
$r\,>\,0$, then this number is clearly zero, because the image of the rational curve in every 
$\mathbb{CP}^2_q$ will precisely be the exceptional divisor and hence it will not intersect any other 
generic point (which will lie outside the exceptional divisor since the points are generic). Hence, we 
conclude that $N_{(d,m)}(r,\theta)$ is zero when $d\,=\,0$, $m\,\leq\,-1$ and $r$ is non-zero, which 
proves equation \eqref{base_case_rec1_pt3}.

Finally, equation \eqref{base_case_rec1_pt6} is obtained by using a geometrically 
non-trivial fact. Recall that a degree $d$ curve in $\cp^2$ has arithmetic genus $g_d\,=\, 
\frac{(d-1)(d-2)}{2}$. If the curve has an ordinary $m$-fold point then the geometric genus
of the curve is given by
\begin{equation*}
g_d - \frac{m(m-1)}{2} - \epsilon
\end{equation*}
where $\epsilon$ is a non-negative integer depending on the other singularities of the curve.
Since the geometric genus is of course a non-negative integer we must have $m \,\leq\, d-1$. Hence, 
in $\mathbb{CP}^2_1$, there are no curves in the class $dL-mE$, if $m \,\geq\, d$. Of course, 
for this reasoning, we are using the correspondence principle that is proved in Section \ref{correspondence}.

Let us now justify the remaining equations. 
The next three equations (namely equations \eqref{base_case_rec2_pt3}, \eqref{base_case_rec2_pt4} and \eqref{base_case_rec2_pt5}) are true, 
because there are no lines 
through $4$ points, $3$ points or $3$ points respectively.

Equation \eqref{base_case_rec2_pt6} (namely $N_{l-e}(2,1)\,=\,1$) is true 
because $N_{l-e}(2,1)$ represents 
the number of lines through two points (this corresponds to the case $r\,=\,2$), 
together with a choice of a point lying in the intersection of this line with the generic
line (corresponding to $\theta\,=\,1$).
Equation \eqref{base_case_rec2_pt7} (namely $N_{l-me}(r,\theta)\,=\,0$ for all $m \,\geq\, 2$) is true
because there are no lines with an $m$-fold point if $m \,\geq\, 2$ 
(again, we are using the correspondence principle).

This concludes the proof of all the base cases of the recursion.
\end{proof}

\subsection{A few remarks regarding intersection theory} 

We will now discuss a few points about intersection theory 
(including the fact as to why the moduli space that is being considered is irreducible). 

Let us start by making a few observations.
Recall that there is a fiber bundle $\pi\,:\,\mathfrak{X}\,\longrightarrow\, \mathbb{P}^2_{\textnormal{Base}}$, 
whose fiber over each point $q$ is $\mathbb{P}^2_q$ (the blow up of $\mathbb{P}^2$ at $q$). Now consider 
a sequence of holomorphic maps $u_n\,:\,\mathbb{P}^1 \,\longrightarrow\, \mathfrak{X}$ such that the image 
of $u_n$ lies in the fiber $\mathbb{P}^2_{q_n}$. Assume that $u_n$ converges to the bubble map $v$. 
We claim that the image of this bubble map will lie in some $\mathbb{P}^2_q$. To see why this is so, assume that 
$q_n$ converges to $q$ (possibly after passing to a subsequence). We can make this assumption
because $\mathbb{P}^2_{\textnormal{Base}}$ 
is compact. We claim that the image of the bubble map $v$ will lie in $\mathbb{P}^2_q$. Let us see why this is so. 
First let us make a simple observation. Suppose $w\,:\,\mathbb{P}^1 \,\longrightarrow\, \mathfrak{X}$ is
a continuous 
map such that the image of $w$ outside a finite set of points $\{y_1,\, \cdots,\, y_k\}$ lies inside a fiber 
$\mathbb{P}^2_q$. Then the image of $\{y_1,\, \cdots,\, y_k\}$ also lies inside $\mathbb{P}^2_q$. This simply follows from the 
fact that the image of $w$ is connected. 

We can now prove the original assertion. Let $v_i\,:\,\mathbb{P}^1 \,\longrightarrow\, \mathfrak{X}$ be one of the 
components of the bubble map $v$. The fact that $u_n$ converges to $v$ means that after composing with a 
sequence of M\"obius transformations $\varphi_n\,:\,\mathbb{P}^1\,\longrightarrow \,\mathbb{P}^1$, the sequence 
$u_n\circ \varphi_n$ converges to $v_i$ uniformly on compact sets outside a finite set of points. This 
follows by directly from the definition of convergence of stable maps \cite[Definition 5.2.1, Page 
114]{McSa}. In particular, $u_n\circ \varphi_n$ converges to $v_i$ pointwise outside a finite set of 
points. But this means that the image of $v_i$ lies inside $\mathbb{P}^2_q$ except for a finite set of 
points. By connectedness, we conclude that the image of $v_i$ lies inside $\mathbb{P}^2_q$. Repeating this 
argument for all the other components of the bubble map tells us that the image of the bubble map $v$ lies 
inside $\mathbb{P}^2_q$.

Let us summarize what we just established. If $u_n$ is a sequence of holomorphic maps whose image lies inside a fiber, 
then the limit of $u_n$ (if it exists) also lies inside a single fiber.

We will now give an alternative (but equivalent) description of our moduli space
$\overline{M}_{0,r}^{\textnormal{Fib}}(\mathfrak{X}, \alpha)$, 
where $\alpha\,:=\, (d,\,m)$. First of all, consider the
natural inclusion map $i_q\,:\, \mathbb{P}^2_q \,\longrightarrow\, \mathfrak{X}$ and define 
\begin{align*} 
\beta_q&\,:= \,d L_q - mE_q \,\in\, H_2(\mathbb{P}^2_q) \ \ \ \textnormal{ and } \ \ \ 
\beta\,:=\, (d L - mE)\cdot a^2 \,\in \,H_2(\mathfrak{X}).
\end{align*} 
Now note that the pushforward of $\beta_q$ in $\mathfrak{X}$ is $\beta$, i.e., 
\begin{align*} 
i_{q_*}(\beta_q) &\, =\,\, \beta. 
\end{align*}
Consider the moduli space $\overline{M}_{0,r}(\mathfrak{X}, \beta)$ of stable maps into $\mathfrak{X}$,
representing the homology class $\beta$. Notice that the curves in this space do not necessary have
to lie inside some fiber $\mathbb{P}^2_q$. 

Let $M_{0,r}^{\textnormal{Fib}}(\mathfrak{X}, \alpha)$ denote the subspace of curves in 
$M_{0,r}(\mathfrak{X}, \beta)$ (the open part of the moduli space)
whose image lies inside some $\mathbb{P}^2_q$. Define 
$\overline{M}_{0,r}^{\textnormal{Fib}}(\mathfrak{X}, \alpha)$ to be the \textit{closure} of 
$M_{0,r}^{\textnormal{Fib}}(\mathfrak{X}, \alpha)$ inside 
$\overline{M}_{0,r}(\mathfrak{X}, \beta)$. 
Since the image of a curve that arises as a 
limit of curves in $M_{0,r}^{\textnormal{Fib}}(\mathfrak{X}, \alpha)$ still lies in a single fiber, 
it follows that $\overline{M}_{0,r}^{\textnormal{Fib}}(\mathfrak{X}, \alpha)$ 
is precisely the space of stable maps whose image lies inside some $\mathbb{P}^2_q$ and that represents 
the class $\beta_q$. This was precisely the original definition of
$\overline{M}_{0,r}^{\textnormal{Fib}}(\mathfrak{X}, \alpha)$. 

Since $\overline{M}_{0,r}^{\textnormal{Fib}}(\mathfrak{X}, \alpha)$ 
is the closure of a variety 
inside 
$\overline{M}_{0,r}(\mathfrak{X}, \beta)$, it defines a cycle. 
We now note that the intersection theory of $\overline{M}_{0,r}(\mathfrak{X}, \beta)$ is well developed 
including the idea of applying the WDVV recursion. This is developed both from the algebro-geometric approach 
(\cite{FuPa}) and from the symplectic geometric approach (\cite{McSa}).
Hence, the intersection numbers defined by equation \eqref{N_alpha_defn} 
and the left hand sides of equations \eqref{1234} and \eqref{1324} 
can be viewed as intersection of cycles inside the ambient space 
$\overline{M}_{0,r}(\mathfrak{X}, \beta)$ (where the intersection theory is well developed). 

One further remark. It was proved by Damiano Testa in \cite{D_Testa} that $\overline{M}_{0,r}(\mathbb{P}^2_q, dL_q-mE_q)$ is an 
irreducible variety of the expected dimension. Hence, 
$\overline{M}_{0,r}^{\textnormal{Fib}}(\mathfrak{X}, \alpha)$ is a fiber bundle over $\mathbb{P}^2_{\textnormal{Base}}$, 
where the fibers are irreducible and the base is also irreducible. Consequently, the total space is irreducible. 
Thus, the moduli space we are considering is irreducible. 

\section{A correspondence result}\label{correspondence}

We now prove the correspondence result. 

\begin{thm}
\label{correspondence_blowup_theorem}
Let $m$ be an integer greater than or equal to $3$. 
Let $N_{\alpha}(r,\,\theta)$ be as defined in equation \eqref{N_alpha_defn}. Then for 
\begin{align}
\alpha&\,=\, (d,\,m), \qquad r\,=\,3d+1-m\qquad \textnormal{and} \qquad \theta\,=\,0,\label{alpha_value_ag}
\end{align}
$N_{\alpha}(r, \,\theta)$ 
gives the number of rational degree $d$ curves in $\mathbb{CP}^2$, passing through $3d+1-m$ generic points and having 
an $m$ fold point singularity.
\end{thm}

\begin{proof}
As usual $\mathbb{CP}^2_q$ denotes the blow-up of $\mathbb{CP}^2$
at the point $q$, and $\pi_q\,:\, \mathbb{CP}^2_q \,
\longrightarrow\, \mathbb{CP}^2$ is the blow-up morphism. 
Let $L_q$ be a divisor corresponding to $\pi_q^* \mathcal{O}_{\mathbb{CP}^2}(1)$ and $E_q$ the
exceptional divisor of the blow-up. 

Let us first prove the correspondence result when the $m$-fold point is at a fixed $q \,\in\, \cp^2$.
If $C\,\subset\, \mathbb{CP}^2$ is a degree $d$ curve with an ordinary $m$-fold point at $q$,
then we denote by $\widetilde{C}$ the strict transform of $C$ in $\mathbb{CP}^2_q$. Note that 
\begin{equation*}
dL_q \,= \,\pi_q^* [C] \,=\, [\widetilde{C}] + m E_q \quad \Rightarrow \quad
[\widetilde{C}] \,=\, dL_q - mE_q.
\end{equation*}
The curve $\widetilde{C}$ intersects $E_q$ at $m$ distinct points transversally. On the other
hand, for any curve $C'\,\subset\, \mathbb{CP}^2_q$ with $[C'] \,=\, dL_q - mE_q$ that
intersects $E_q$ at $m$ distinct points, $C \,=\, \pi_q(C')$ is a degree $d$ curve in 
$\mathbb{CP}^2$ with an ordinary $m$-fold point at $q$.
Hence, the set of rational degree $d$ curves in $\cp^2$ with an $m$-fold point at 
$q$ and passing through $3d-1-m$ generic points is in bijection with the set of
rational curves in $\cp^2_q$ in the class $dL_q-mE_q$ passing through $3d-1-m$ generic points 
which intersect $E_q$ at $m$-distinct points. 

\begin{center}
\includegraphics[scale=1]{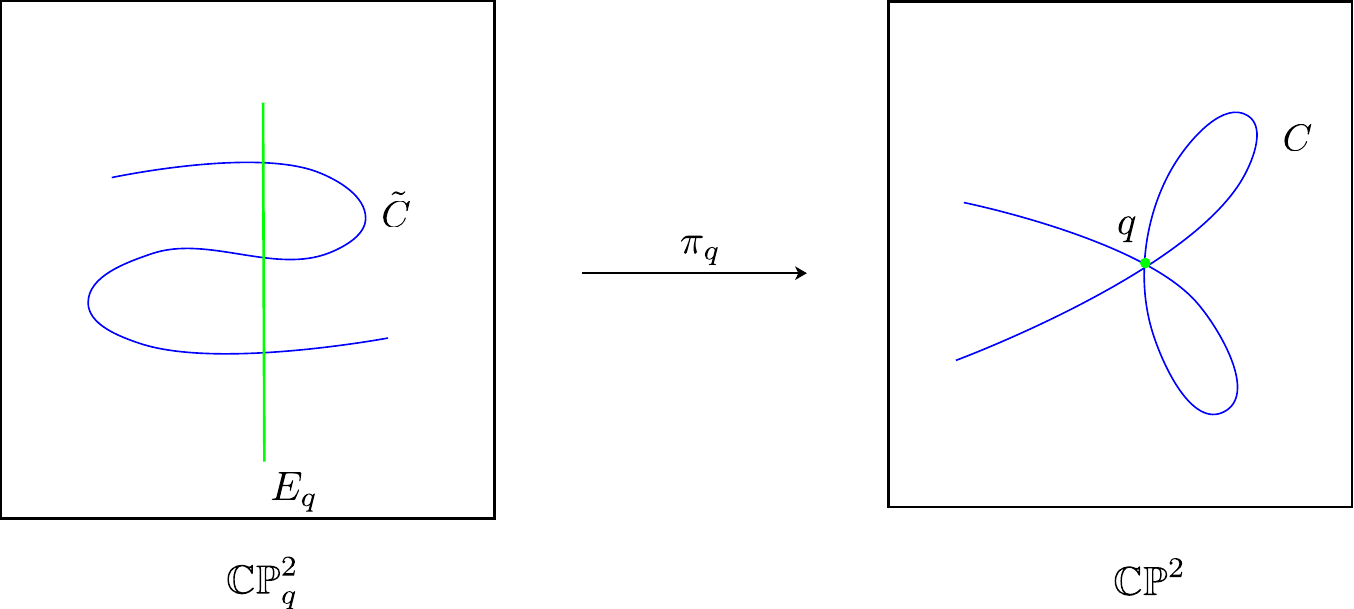}
\end{center}

Let $\mbar_{0,n}(\cp^2_q,\, \beta_q)$ denote the moduli space of genus zero stable maps 
to $\cp^2_q$ with $n$ marked points representing $\beta_q \,\in\, H_2(\cp^2_q,\, \mathbb{Z})$.
The Gromov-Witten invariants of $\cp^2_q$ are enumerative, hence for $\beta_q \,=\, dL_q - mE_q$
and $r \,=\, 3d-1-m$
\begin{equation*}
\gw{\cp^2_q,\, \beta_q}{0,r}{(L_q^2)^r} \,=\, 
\left[\mbar_{0, 3d-1-m}(\cp^2_q,\,\beta_q) \right] \cdot \left( \prod_{i=1}^{3d-1-m} 
\ev_i^*(L_q^2) \right) 
\end{equation*}
counts the number of genus $0$ curves in $\cp^2_q$ passing through $3d-1-m$ generic points 
representing class $dL_q-3E_q$. Of course these curves may not all intersect $E$ at 
$m$ distinct points. To ensure this we add $m$ marked points and require them to lie on
$E_q$; then this number is given by 
\begin{equation*}
\gw{\cp^2_q, \beta_q}{0,3d-1}{(E_q)^m,\, (L_q^2)^{3d-1-m})}\,=\,
\left[\mbar_{0, 3d-1}(\cp^2_q,\, \beta_q)\right] \cdot \left( \prod_{i=1}^m \ev_i^*(E_q) 
\right) \cdot \left( \prod_{i=m+1}^{3d-1} \ev_i^*(L_q^2) \right).
\end{equation*}
Now by the divisor axiom
\begin{equation*}
\gw{\cp^2_q, \beta_q}{0,3d-1}{(E_q)^m,\, (L_q^2)^{3d-1-m})} \,=\, 
m^m \gw{\cp^2_q, \beta_q}{0,r}{(L_q^2)^r}. 
\end{equation*}
We only want to count the curves which intersect $E_q$ at distinct points. 
Suppose the first two marked points are equal and both lie on $E$, we can count this number as 
follows: consider the diagonal map
$\Delta\,:\, \cp^2_q \,\longrightarrow\, \cp^2_q \times \cp^2_q$; then $\Delta(E_q)$ has
codimension $3$ in 
$\cp^2_q \times \cp^2_q$, and hence
\begin{equation*}
 \left[\mbar_{0, 3d-1}(\cp^2_q, \,\beta_q)\right] \cdot \Big((\ev_1\times\ev_2)^* \Delta(E_q) 
 \Big) \cdot
 \left( \prod_{i=3}^m \ev_i^*(E_q) \right) \cdot
 \left( \prod_{i=m+1}^{3d-1} \ev_i^*(L_q^2) \right)\,=\,0;
\end{equation*}
this is because $\dim \mbar_{0,3d-1}(\cp^2_q,\beta_q)\, =\, 6d-2-m$ whereas the cycle
$$  \Big((\ev_1\times\ev_2)^* \Delta(E_q) 
    \Big) \cdot
    \left( \prod_{i=3}^m \ev_i^*(E_q) \right) \cdot
    \left( \prod_{i=m+1}^{3d-1} \ev_i^*(L_q^2) \right)$$ 
has codimension $3 + (m-2) + 2(3d-1-m) \,=\, 6d - 1 - m$.
Hence, $\gw{\cp^2_q, \beta_q}{0,r}{(L_q^2)^r}$ in fact counts the curves which intersect 
$E_q$ at distinct points. This proves the result when the $m$-fold point is at a fixed point.

Now consider the general case. First we shall show that $N_\alpha(r,\,0)$ only counts the 
genus zero $r$ pointed stable maps into a fiber $\cp^2_q$ of $\X$ which intersect $E_q$ at
$m$ distinct points.

Let $r \,=\, 3d+1-m$ and $\alpha \,=\, (d,\,m) \,\in\, \mathbb{Z}^2$. 
It can be shown that
\begin{equation}\label{eq:divisor}
\left[\mbarfib{r+1}(\X,\alpha)\right]\cdot \prod_{i=1}^r \ev_i^*(L^2) \cdot \ev_{r+1}^*(E) \,=\, 
m \left[\mbarfib{r}(\X,\alpha)\right]\cdot \prod_{i=1}^r \ev_i^*(L^2) \,=\, m N_\alpha(r,\,0).
\end{equation} 
To prove this we proceed as in \cite[Section 7, Lemma 14]{FuPa}. Consider the 
forgetful morphism $$\pi_{r+1}\,:\, \mbarfib{r+1}(\X,\alpha) \,\longrightarrow\, 
\mbarfib{r}(\X,\alpha).$$ The restriction 
\begin{equation*}
\pi_{r+1} \,\,:\,\, \ev_{r+1}^{-1}(E)\,\,\longrightarrow\,\, \mbarfib{r}(\X,\alpha)
\end{equation*}
is a finite morphism with a generic fiber having cardinality $m$, hence 
$(\pi_{r+1})_* [\ev_{r+1}^{-1}(E)]\,=\, m \left[\mbarfib{r}(\X,\alpha)\right]$. 
Now $\eqref{eq:divisor}$ 
follows from
\begin{equation*}
\left[\mbarfib{r+1}(\X,\alpha)\right]\cdot \prod_{i=1}^r \ev_i^*(L^2) \cdot \ev_{r+1}^*(E) = 
[\ev_{r+1}^{-1}(E)] \cdot \prod_{i=1}^r \ev_i^*(L^2) =
(\pi_{r+1})_*[\ev_{r+1}^{-1}(E)] \cdot \prod_{i=1}^r \ev_i^*(L^2)
\end{equation*}
because $\ev_i\circ \pi_{r+1} \,=\, \ev_i$ for $i \,=\,1,\,\cdots,\,r$. Thus we see that 
\begin{equation*}
\left[\mbarfib{r+m}\right]\cdot \prod_{i=1}^m \ev_i^*(E) \cdot 
\prod_{i=m+1}^{r+m} \ev_{i}^*(L^2) \,=\, m^m N_\alpha(r,\,0).
\end{equation*}
This number counts the curves with last $m$ marked points on $E$. Now suppose the first two
of those marked points on $E$ coincide. We have a map 
$$\ev_1 \times \ev_2 \,:\, \mbarfib{r+m}(\X,\alpha)
\,\longrightarrow\, \X \times_{\cp^2} \X$$ and let 
$\Delta\,:\, \X \,\longrightarrow\, \X \times_{\cp^2} \X$ be the diagonal map. For some $\mu \,\in
\,\mbarfib{r+m}(\X,\alpha)$ if
$\ev_1(\mu)\, =\, \ev_2(\mu)\,\in\, E$, then $(\ev_1\times \ev_2)(\mu) \,\in\, \Delta(E)$. As before
$\Delta(E)$ has codimension three in $\X \times_{\cp^2} \X$. By dimension count we see that
\begin{equation*}
\left[\mbarfib{r+m}(\X,\alpha)\right] \cdot (\ev_1\times\ev_2)^*(\Delta(E))
\cdot \prod_{i=3}^m \ev_i^*(E) \cdot \prod_{i=m+1}^{r+m} \ev_{i}^*(L^2) \,=\, 0.
\end{equation*}
This is precisely the number of stable maps with first $m$ marked points on $E$ and such 
that the first two marked points coincide. We thus see that $N_\alpha(r,\,0)$ only counts
stable maps which intersect the exceptional divisor of any fiber at distinct points.
For every such curve after blowing down the exceptional divisor we get a curve with 
has an ordinary $m$-fold point.
\end{proof}

\section{Low degree checks}\label{low_deg_check}

\subsection{Rational curves with a choice of a node}

We now make some low degree checks. Let us start by recapitulating some numbers from 
Kontsevich's recursion formula. Denote by $n_d$ the number for rational curves of degree $d$ 
in $\mathbb{CP}^2$ that pass through $3d-1$ generic points. The first few values of $n_d$ 
obtained from the recursion formula are: 
\begin{center}
\begin{tabular}{|c|c|c|c|c|c|c|} 
\hline
$d$ &$3$ &$4$ & $5$ & $6$ & $7$ & $8$ \\
\hline
$n_{d}$ & $12$ & $620$ & $87304$ & $26312976$ & $14616808192$ & $13525751027392$ \\ 
\hline
\end{tabular}
\captionof{table}{\textnormal{~~}} \label{t_kontsevich}
\end{center}

We will now subject our formula to some low degree checks. Let $\alpha\,=\,(d,\,m)$.
For $\theta\,=\,0$, $m\,=\,2$ and $r\,=\,3d-1$, the values obtained are: 
\begin{center}
\begin{tabular}{|c|c|c|c|c|c|c|} 
\hline
$(d,m, \theta,r)$ &$(3,2,0,8)$ &$(4,2,0,11)$ & $(5,2,0,14)$ & $(6,2,0,17)$ & $(7,2,0,20)$ & $(8,2,0,23)$ \\
\hline
$N_{\alpha}(r, \theta)$ & $12$ & $1860$ & $523824$ & $263129760$ & $219252122880$ & $284040771575232$ \\ 
\hline
\end{tabular}
\captionof{table}{\textnormal{~~}} \label{t1}
\end{center}
The values for $N_{\alpha}(r, \,\theta)$ listed in Table \ref{t1}
are precisely equal to the number of rational degree $d$ curves in $\mathbb{CP}^2$ 
passing through $3d-1$ generic points together with a choice of a node. Since a rational degree $d$ curve passing through $3d-1$ 
generic points has 
precisely $\frac{(d-1)(d-2)}{2}$ nodal points, we expect the answer to be 
\[ \frac{(d-1)(d-2)}{2} n_d.\]
The numbers 
in Table \ref{t1} confirm this.

Next, set $\alpha\,=\,(d,\,m)$, $\theta\,=\,1$, $m\,=\,2$ and $r\,=\,3d-2$. 
In this case, note that 
$N_{\alpha}(r, \,\theta)$ is precisely equal to the number of rational curves of degree $d$
in $\mathbb{CP}^2$ passing through $3d-2$ generic points, with a nodal point lying on a line. 
The values obtained are: 
\begin{center}
\begin{tabular}{|c|c|c|c|c|c|c|} 
\hline
$(d,m, \theta,r)$ &$(3,2,1,7)$&$(4,2,1,10)$ & $(5,2,1,13)$ & $(6,2,1,16)$ & $(7,2,1,19)$ & $(8,2,1,22)$ \\
\hline
$N_{\alpha}(r, \theta)$ & $6$ & $768$ & $181320$ & $78076800$ & $56831124000$ & $65305557682176$ \\ 
\hline
\end{tabular}
\captionof{table}{\textnormal{~~}} \label{t2}
\end{center}
The numbers listed
in Table \ref{t2} 
agree with the ones obtained 
in \cite[p.~695, Lemma 7.5]{zin-triple}. To see this, note that
\cite[Lemma 7.5]{zin-triple} says that the desired number is equal to 
\begin{align*}
\frac{1}{2}[\overline{M}_{0,1}(\mathbb{CP}^2,\, d)]\cdot
\Big((2d-3)\textnormal{ev}^*(a^2)-\textnormal{ev}^*(a) \psi\Big)\cdot \mathcal{H}^{3d-2}, 
\end{align*}
where $a$ denotes the class of a line in $\mathbb{CP}^2$ and $\mathcal{H}$ denotes the divisor 
corresponding to the space of curves whose image passes through a generic point while $\psi$ denotes the 
first Chern class of the universal cotangent bundle. Using the property of $\psi$ as given by \cite[p.~96, 
Equation 5.5]{BVSRM} we get the numbers in the second table.

Next, we list the values of $N_{(d,m)}(r,\, \theta)$ when $\theta\,=\,2$, $m\,=\,2$ and $r\,=\,3d-3$: 
\begin{center}
\begin{tabular}{|c|c|c|c|c|c|c|} 
\hline
$(d,m, \theta,r)$ &$(3,2,2,6)$&$(4,2,2,9)$ & $(5,2,2,12)$ & $(6,2,2,15)$ & $(7,2,2,18)$ & $(8,2,2,21)$ \\
\hline
$N_{\alpha}(r, \theta)$ & $1$ & $96$ & $18132$ & $6506400$ & $4059366000$ & $4081597355136$ \\
\hline
\end{tabular}
\captionof{table}{\textnormal{~~}} \label{t3}
\end{center}
In Table \ref{t3}, $N_{\alpha}(r,\, \theta)$
is equal to the number of rational curves of degree $d$ in $\mathbb{CP}^2$ 
passing through $3d-3$ generic points, with a nodal point located at a fixed point. 
These numbers agree with the third line of the table displayed in \cite[p.~45, Example 8.1]{Gath_blow_up}. 
In principle, it should also be possible to modify \cite[p.~695, Lemma 7.5]{zin-triple}
and verify the numbers in Table \ref{t3}. 

\subsection{Rational curves with a triple point} 

Let us now consider the result of \cite{zin-triple}, where the author enumerates rational curves in $\mathbb{CP}^2$ 
with a triple point. We will obtain those numbers using the result of this paper. 

Set $\alpha\,=\,(d,\,m)$.
For $\theta\,=\,0$, $m\,=\,3$ and $r\,=\,3d-2$, the values we get for $N_{\alpha}(r, \theta)$ are: 
\begin{center}
\begin{tabular}{|c|c|c|c|c|c|} 
\hline
$(d,m, \theta,r)$ & $(4,3,0,10)$ & $(5,3,0,13)$ & $(6,3,0,16)$ & $(7,3,0,19)$ & $(8,3,0,22)$ \\
\hline
$N_{\alpha}(r, \theta)$ & $60$ & $56400$ & $49177440$ & $56784765120$ & $91466185097280$ \\ 
\hline
\end{tabular}
\captionof{table}{\textnormal{~~}} \label{t4}
\end{center}
The values listed for $N_{\alpha}(r, \,\theta)$ in Table \ref{t4} 
are precisely equal to the number of rational degree curves of $d$ in $\mathbb{CP}^2$ that pass through $3d-2$ 
generic points and have a triple point. 
The values listed in the Table \ref{t4} 
agree precisely with the values obtained 
in \cite{zin-triple} (displayed on Table $1$, Page $695$).

For the convenience of the reader, we list the values of $N_{\alpha}(r, \,\theta)$ when $\alpha\,=\,(d,\,m)$, 
$\theta\,=\,1$, $m\,=\,3$ and $r\,=\,3d-3$:
\begin{center}
\begin{tabular}{|c|c|c|c|c|c|} 
\hline
$(d,m, \theta,r)$ & $(4,3,1,9)$ & $(5,3,1,12)$ & $(6,3,1,15)$ & $(7,3,1,18)$ & $(8,3,1,21)$ \\
\hline
$N_{\alpha}(r, \theta)$ & $12$ & $9600$ & $7221096$ & $7307731200$ & $10461017642880$ \\ 
\hline
\end{tabular}
\captionof{table}{\textnormal{~~}} \label{t5}
\end{center}
We also list the values of 
$N_{\alpha}(r, \theta)$ when $\alpha\,=\,(d,\,m)$, 
$\theta\,=\,2$, $m\,=\,3$ and $r\,=\,3d-4$:
\begin{center}
\begin{tabular}{|c|c|c|c|c|c|} 
\hline
$(d,m, \theta,r)$ & $(4,3,2,8)$ & $(5,3,2,11)$ & $(6,3,2,14)$ & $(7,3,2,17)$ & $(8,3,2,20)$ \\
\hline
$N_{\alpha}(r, \theta)$ & $1$ & $640$ & $401172$ & $347987200$ & $435875735120$ \\ 
\hline
\end{tabular}
\captionof{table}{\textnormal{~~}} \label{t6}
\end{center}
The number $N_{(d,m)}(r, \,\theta)$ when 
$\theta\,=\,1$, $m\,=\,3$ and $r\,=\,3d-3$ gives us the number of rational curves of degree $d$
in $\mathbb{CP}^2$ that pass through $3d-3$ generic points and have a triple point lying on a line. Similarly, 
$N_{(d,m)}(r,\,\theta)$ when $\theta\,=\,2$, $m\,=\,3$ and $r\,=\,3d-4$ gives us the number of rational
curves of degree $d$ in $\mathbb{CP}^2$ 
that pass through $3d-4$ generic points and have a triple point lying on a point.
The numbers given in Table \ref{t6} agree
with the fourth line of the table displayed in \cite[p.~45, Example 8.1]{Gath_blow_up}.
In principle, it should be possible to modify the result of \cite{zin-triple}
and obtain the numbers given in Tables \ref{t5} and \ref{t6}. 

\subsection{Rational curves with an $m$-fold point}

Again, we set $\alpha\,=\,(d,\,m)$. 
For $\theta\,=\,0$, $m\,=\,d-1$ and $r\,=\,2d+2$; the values we get are
\begin{center}
\begin{tabular}{|c|c|c|c|c|c|c|} 
\hline
$(d,m, \theta, r)$ & $(4,3,0,10)$ &$(5,4,0, 12)$ & $(6,5,0,14)$ & $(7,6,0,16)$ & $(8,7,0,18)$ & $(9,8,0,20)$ \\
\hline
$N_{\alpha}(r, \theta)$ & $60$ &$180$ & $420$ & $840$ & $1512$ & $2520$ \\ 
\hline
\end{tabular}
\end{center}
Similarly, for $\theta\,=1$, $m\,=\,d-1$ and $r\,=\,2d+1$, the values we get are
\begin{center}
\begin{tabular}{|c|c|c|c|c|c|c|} 
\hline
$(d,m, \theta, r)$ & $(4,3,1,9)$ &$(5,4,1, 11)$ & $(6,5,1,13)$ & $(7,6,1,15)$ & $(8,7,1,17)$ & $(9,8,1,19)$ \\
\hline
$N_{\alpha}(r, \theta)$ & $12$ &$20$ & $30$ & $42$ & $56$ & $72$ \\ 
\hline
\end{tabular}
\end{center} 
Finally, for $\theta\,=\,2$, $m\,=\,d-1$ and $r\,=\,2d$, the values are
\begin{center}
\begin{tabular}{|c|c|c|c|c|c|c|} 
\hline
$(d,m, \theta, r)$ & $(4,3,2,8)$ &$(5,4,2, 10)$ & $(6,5,2,12)$ & $(7,6,2,14)$ & $(8,7,2,16)$ & $(9,8,2,18)$ \\
\hline
$N_{\alpha}(r, \theta)$ & $1$ &$1$ & $1$ & $1$ & $1$ & $1$ \\ 
\hline
\end{tabular}
\end{center}
These values are as expected. To see why, note that the numbers displayed in the above table 
are the number of rational curves of degree $d$ passing through $2(d+1)-\theta$ generic points that 
have a $(d-1)$ fold singular point lying on a cycle $a^{\theta}$. Note that a
smooth degree $d$-curve in $\mathbb{CP}^2$ 
has genus $\frac{(d-1)(d-2)}{2}$. Furthermore, we note that an $m$-fold singular point 
contributes $\frac{m(m-1)}{2}$ to the genus. Consequently, a $(d-1)$-fold singular point contributes 
$\frac{(d-1)(d-2)}{2}$ to the genus. Thus $N_{d, d-1}(2d+2, \,\theta)$ can also be interpreted as the 
number of curves in the linear system $H^0(\mathcal{O}_{\mathbb{CP}^2}(d))$
passing through 
$2(d+1)-\theta$ generic points and having a $(d-1)$-fold singular point lying on the cycle $a^{\theta}$. 
For example, when $d\,=\,4$, then $N_{4, 3}(10,\,0)$
is the number of quartics through $10$-generic points, having a triple point. Let us see how to solve this question.

Let $\mathcal{D} \,:=\, \mathbb{P}(H^0(\mathcal{O}_{\mathbb{CP}^2}(d)))$ denote the space of all
homogeneous degree $d$ polynomials in three variables up to scaling. This is a projective space of dimension 
$\frac{d(d+3)}{2}$. We will interpret our desired numbers as intersection numbers inside the space 
$\mathcal{D}\times \mathbb{CP}^2$. Denote the hyperplane classes of $\mathcal{D}$ and 
$\mathbb{CP}^2$ by $y$ and $a$ respectively. Let $f \,\in\, H^0(\mathcal{O}(d))$ be a homogeneous degree 
$d$ polynomial. Given $m\,\geq\, 2$, define the cycle $S_m$ as
\begin{align*}
S_m &\,:=\, \{ ([f],\, q) \,\in\, \mathcal{D}\times \mathbb{CP}^2\,\mid\, f(q)
\,=\,0 , ~~\nabla f|_{q}\,=\,0,
~~\nabla^2 f|_q \,=\,0,\,\cdots ,\,\nabla^{m-1} f|_q \,=\,0 \}. 
\end{align*}
Note that $ S_m$ is a cycle inside $\mathcal{D}\times \mathbb{CP}^2$. We will now interpret the evaluation maps and the 
derivatives up to order $m-1$ as sections of an appropriate bundle. Note that the evaluation map 
\[ (f,\, q) \,\,\longrightarrow\,\, f(q) \]
induces a section of the line bundle 
\begin{align*}
\mathbb{V}_{0}&\,\,:=\, \,\gamma_{\mathcal{D}}^* \otimes (\gamma^*_{\mathbb{CP}^2})^{d},
\end{align*}
where $\mathcal{\gamma}_{\mathcal{D}}$ and $\mathcal{\gamma}_{\mathbb{CP}^2}$ 
denote the tautological line bundles over $\mathcal{D}$ and $\mathbb{CP}^2$ respectively while 
$\mathcal{\gamma}_{\mathcal{D}}^*$ and $\mathcal{\gamma}_{\mathbb{CP}^2}^*$ respectively denote their duals. 
Note that 
\[c_1(\gamma_{\mathbb{CP}^2}^*)\,=\,a \ \ \ \textnormal{ and }\ \ \ \qquad c_1(\gamma_{\mathcal{D}}^*)
\,=\, y. \] 
We now note that the first derivative 
\[ (f,\, q)\,\, \longrightarrow\,\, \nabla f|_q \]
induces a section of the rank two vector bundle 
\begin{align*}
\mathbb{V}_{1}&\,\,:=\,\, \gamma_{\mathcal{D}}^* \otimes T^*\mathbb{CP}^2\otimes
(\gamma_{\mathbb{CP}^2}^*)^d.
\end{align*}
Similarly, the second derivative induces a section of the rank three vector bundle 
\begin{align*}
\mathbb{V}_{2}&\,\,:=\,\, \gamma_{\mathcal{D}}^* \otimes 
\textnormal{Sym}^2(T^*\mathbb{CP}^2)\otimes (\gamma_{\mathbb{CP}^2}^*)^d.
\end{align*}
In general, the $k^{\textnormal{th}}$-derivative induces a section of the rank $(k+1)$ vector bundle given by 
\begin{align*}
\mathbb{V}_{k}&\,\,:=\,
\,\gamma_{\mathcal{D}}^* \otimes \textnormal{Sym}^k(T^*\mathbb{CP}^2)\otimes (\gamma_{\mathbb{CP}^2}^*)^d.
\end{align*}
Hence, as a cycle, 
\begin{align*}
[S_m]& \,\,= \,\,e(\mathbb{V}_0)\cdot e(\mathbb{V}_1) \cdot e(\mathbb{V}_2) \cdot \ldots \cdot e(\mathbb{V}_{m-1}),
\end{align*}
were $e$ denotes the Euler class. The Euler class of each of the
vector bundles can be computed via the splitting principle. 
Let us define
\begin{align*}
c_1&\,:=\, y+da, \qquad \alpha_1\,:=\, -3a \qquad \textnormal{and} \qquad \beta_2\,:=\, 3 a^2.
\end{align*}
Then, via the splitting principle, the 
Euler class of the first few $\mathbb{V}_k$ are explicitly given by: 
\begin{align*}
e(\mathbb{V}_0) & \,=\, c_1, \qquad e(\mathbb{V}_1) \,=\, c_1^2 +c_1 \alpha_1 + \beta_2, \ \ \  
e(\mathbb{V}_2) \,= \,c_1^3 +3 c_1^2 \alpha_1 + c_1(2\alpha_1^2 + 4\beta_2), \\ 
\qquad e(\mathbb{V}_3) & \,=\, c_1^4 +6 c_1^3 \alpha_1 + c_1^2(11\alpha_1^2 + 10\beta_2), \ \ \  
e(\mathbb{V}_4)  \,=\, c_1^5 +10 c_1^4 \alpha_1 + c_1^3(35\alpha_1^2 + 20\beta_2), \\ 
e(\mathbb{V}_5)  &\, =\, c_1^6 +15c_1^5 \alpha_1 + c_1^4(85\alpha_1^2 + 35\beta_2), \ \ \  
e(\mathbb{V}_6)\,=\, c_1^7 +21c_1^6 \alpha_1 + c_1^5(175\alpha_1^2 + 56\beta_2) \ \ \ \textnormal{and} \\ 
e(\mathbb{V}_7)  & \,=\, c_1^8 +28c_1^7 \alpha_1 + c_1^6(322\alpha_1^2 + 84\beta_2). 
\end{align*}
Given $d$ and $m$, define 
\begin{align*}
j &\,\,:=\,\, \frac{d(d+3)}{2} - \Big(\frac{m^2+m-4}{2}\Big). 
\end{align*}
The number of curves of degree $d$ passing through 
$j-\theta$ generic points, and having an $m$-fold singular point lying on the cycle $a^{\theta}$, is
given by the number 
\begin{align*}
[S_m]\cdot y^{j} \cdot a^{\theta}. 
\end{align*}
Performing this computation for $m\,=\,d-1$ and $d\,=\,4,\, 5,\, 6,\, \cdots,\, 9$ gives us
precisely the numbers obtained in the three tables. 

\section{Acknowledgement} 
We are grateful to the referee for giving us valuable feedback to improve our manuscript.

\bibliography{Degreed} 
\bibliographystyle{siam}

\end{document}